\documentclass[12pt,leqno]{amsart}
\usepackage{pifont}
\usepackage{stmaryrd}
\usepackage{dsfont}
\usepackage{color}
\usepackage{mathrsfs}
\usepackage{amsmath}
\usepackage{amssymb}
\usepackage{hyperref}
\usepackage{bookmark}
\usepackage[bottom]{footmisc}
\usepackage{verbatim}
\usepackage{extarrows}
\usepackage{mathtools}

\numberwithin{equation}{section}
\newtheorem{thm}{Theorem}[section]
\newtheorem{lem}[thm]{Lemma}

\newtheorem{prop}[thm]{Proposition}

\newtheorem{defi}[thm]{Definition}
\newtheorem{rmk}[thm]{Remark}

\newcommand{\End}{\operatorname{End}}

\newcommand{\wt}{\operatorname{wt}}
\newcommand{\xqz}[1]{\lfloor #1 \rfloor}
\newcommand{\Res}{\operatorname{Res}}

\def\Z{\mathbb{Z}}
\def\N{\mathbb{N}}

\def\C{\mathbb{C}}

\allowdisplaybreaks

\begin{document}

\title{Refining twisted bimodules  associated to VOAs}

\author{Shun Xu}

\address{School of Mathematical Sciences, Tongji University, Shanghai, 200092, China}

\email{shunxu@tongji.edu.cn}

\author{Jianzhi Han}

\address{School of Mathematical Sciences, Key Laboratory of Intelligent Computing and
Applications (Ministry of Education), Tongji University, Shanghai, 200092, China}

\email{jzhan@tongji.edu.cn}



\begin{abstract}
Let $V$ be a vertex operator algebra and $g$ an automorphism of $V$ of finite order $T$.
For any $m, n \in(1/T) \mathbb N$, an $A_{g,n}(V)\!-\!A_{g,m}(V)$ bimodule $A_{g,n, m}(V)=V/O_{g,n,m}(V)$ was defined by Dong and Jiang, where $O_{g,n,m}(V)$ is the sum of three certain subspaces $O_{g,n, m}^{\prime}(V), O_{g,n, m}^{\prime \prime}(V)$ and $O_{g,n, m}^{\prime \prime \prime}(V)$.
In this paper, we show that $O_{g,n, m}(V)=O_{g,n, m}^{\prime}(V)$.
\end{abstract}

\maketitle

\section{Introduction}
Let $V$ be a vertex operator algebra (VOA) (see \cite{B,FLM1}). The Zhu algebra $A(V)=V/O(V)$ \cite{Z} serves as a fundamental tool in the representation theory of $V$. Building on this idea of constructing an associative algebra from $V$, Dong, Li and Mason \cite{DLM3} introduced a family of associative algebras $A_n(V)=V/O_n(V)$, where  $A(V)=A_0(V)$.  Further generalizations include twisted Zhu algebras $A_{g,n}(V)=V/O_{g,n}(V)$ for $n\in(1/T)\N$, associated with an automorphism $g$ of $V$ of finite order $T$  (see  \cite{DLM1,DLM2}). Notably,  $A_{n}(V)$ coincides with $A_{1,n}(V)$.

These twisted Zhu algebras play a pivotal role in studying admissible $g$-twisted $V$-modules $M=\bigoplus_{k\in(1/T)\N}M(k)$.  For each $k\in(1/T)\N$ with $k\le n$, $M(k)$ becomes an $A_{g,n}(V)$-module.  When $V$ is rational,  $A_{g,n}$ decomposes as
$\bigoplus_{M}\bigoplus_{(1/T)\N\ni k\le n}{\rm End}\, M(k)$, where $M$ ranges over all inequivalent irreducible admissible $g$-twisted $V$-modules (see \cite{DJ1}). Importantly, the equivalence classes of such modules are in one-to-one correspondence with certain irreducible $A_{g,n}(V)$-modules (see \cite{DLM2}).

Bimodules were initially introduced to determine the fusion rules  (see  \cite{FZ1,L}). A related class of bimodules, $A_{g,n,m}(V)$, was constructed by Dong and Jiang \cite{DJ1} to study  spaces  of the form $\sum_{M}\bigoplus_{0\le l\le{\rm min}\{m,n\}}{\rm Hom}_{\C}(M(m-l), M(n-l))$.
These generalize the twisted Zhu algebras in the sense that $A_{g,n}(V)=A_{g,n,n}(V)$.
A key advantage  of $A_{g,n,m}(V)$
lies in their explicit use for constructing admissible $g$-twisted $V$-modules of Verma type (see \cite{DJ1}).

There is a deep connection between $A_{g,n,m}(V)$ and the universal enveloping algebra $U(V[g])$ (see \cite{FZ1}). All twisted Zhu algebras are special cases of these bimodules, and the entire family of $A_{g,n,m}(V)$ for $m,n\in(1/T)\N$ can be integrated into $U(V[g])$. More precisely, each bimodule $A_{g,n,m}(V)$ can be realized as a quotient of $U(V[g])$ \cite{HXX} (see also \cite{H,HX,He}).

The definition of  $A_{g,n,m}(V)$
 involves the subspace $O_{g,n,m}(V)$, which, for some techinical reasons,
  is defined as the sum of three components: $O^\prime_{g,n,m}(V)$, $O^{\prime\prime}_{g,n,m}(V)$ and $O^{\prime\prime\prime}_{g,n,m}(V)$, where $O^\prime_{g,n,m}(V)$ is analogous to $O_{g,n}(V)$. An alternative definition of $O_{g,n,m}(V)$ arises from representation theory (see \cite{H,HXX}). It has been shown that $O_{g,n,n}(V)=O^\prime_{g,n,n}(V)$ for all $n\in(1/T)\N$  (see \cite{DJ1}), leading to  the  conjecture that $O_{g,n,m}(V)=O^\prime_{g,n,m}(V)$ for any $m,n\in(1/T)\N$. Recent progress on this conjecture established  \cite{HXX}:\[O_{g, n, m}(V)=\bigoplus_{s\,\not\equiv\, {\bar{m}-\bar{n}} \bmod T} V^s+
    L_{n,m}(V)+O_{g, n, m}^{\prime\prime}(V),\]
    where $\bigoplus_{s\,\not\equiv\, {\bar{m}-\bar{n}} \bmod T} V^s+
    L_{n,m}(V)$ is a subspace of $O_{g,n,m}^\prime(V)$ (see Section 2). This conjecture is completely resolved in the present paper.

\begin{thm}\label{thm1.1}
For any $n,m\in(1/T)\N$, we have
\[
    O_{g,n,m}(V)=O^\prime_{g,n,m}(V).
\]
\end{thm}

 The proof of this theorem relies crucially on the twisted regular representation of $V$ \cite{LS1}, which focuses on constructing a new twisted module on the dual space of a given twisted module. Zhu \cite{ZY} also employed this representation to  prove Theorem \ref{thm1.1}, demonstrating through direct computation that    each of $O^{\prime\prime}_{g,n,m}(V)$ and $O^{\prime\prime\prime}_{g,n,m}(V)$ is a subspace of $O^{\prime}_{g,n,m}(V).$
  In our framework, the representation $(\mathfrak D(V), Y^R)$  that we utilize corresponds to only half part of the  twisted  regular representation. And  our approach differs by exploiting  the relationship between $O_{g,n,m}(V)$ and $\Omega_m(M, Y_M)$ (see Lemma \ref{lem3.1}),  resulting in a more concise proof.

This paper is organized as follows. In Section 2, we first recall the construction of $A_{g,n,m}(V)$ and then  use some known results to  give an equivalent condition  for $O_{g,n,m}(V)=O^\prime_{g,n,m}(V).$ In Section 3, we review  the twisted regular representation $(\mathfrak{D}(V), Y^R)$ and establish the inclusion  $\left(A_{g,n, m}^{\prime}(V)\right)^*\subseteq\Omega_{n}\left(\mathfrak{D}(V),Y^R\right)$ (see Proposition \ref{ex-porp}).  Section 4  is devoted to the proof of Theorem \ref{thm1.1}.

\section{$O_{g,n,m}(V)$ and $\Omega_n(M,Y_M)$}

Let $(V,Y,\mathbf{1},\omega)$ be a vertex operator algebra. For any $n\in\Z$, elements in $V_n$ are called homogeneous  and if $u\in V_n$  we define $\operatorname{wt}u=n$. As a convention, when  $\operatorname{wt}u$ appears we always assume that $u$ is homogeneous. Write $L(n)=\omega_{n+1}$ for any $n\in\Z$.

\begin{defi}
A linear isomorphism $g$ of $V$ is called an automorphism of $V$ if
$$
g(\mathbf{1})=\mathbf{1},\ g(\omega)=\omega \ \text { and } \ g(Y(u, z) v)=Y(g(u), z) g(v) \ \text { for } u, v \in V
$$
\end{defi}
   We fix $g$ to be an automorphism of $V$ of finite order $T$. Then $V$ has the following decomposition with respect to $g$:
$$
V=\bigoplus_{r=0}^{T-1} V^r, \quad \text { where } V^r=\left\{v \in V \mid g v=\zeta^r v\right\} \text{ with } \zeta=e^{-2\pi\sqrt{-1}/T}.
$$

Let us first recall the construction of the $A_{g,n}(V)\!-\!A_{g,m}(V)$-bimodule $A_{g,n,m}(V)$. For any $n \in(1 / T) \N,$ there exists an $\bar{n} \in \{0,1,\cdots,T-1\}$ such that $n=\lfloor n\rfloor+\bar{n} / T$, where $\lfloor\cdot\rfloor$ is
the floor function.  For $0\leq r\leq T-1$,  define $\delta_{i}(r)=1$ if $r \leq i \leq T-1$ and $\delta_{i}(r)=0$ if $i<r$; and set $\delta_{i}(T)=0$ for $0\le i\le T-1$.
For any formal variable $x$ and any $\alpha \in \mathbb{R}$, define
$$
(1+x)^\alpha=\sum_{i \in \N} \binom{\alpha}{i} x^i,
$$
where $\binom{\alpha}{i}=\frac{\alpha(\alpha-1) \cdots(\alpha-i+1)}{i!}$ for $i>0$ and $\binom{\alpha}{0}=1$.

Now for  $u \in V^r, v\in V$ and $m,n,p \in (1/T)\N$, define the product $ *_{g, m, p}^{n}$ on $V$ as follows:
\begin{align*}
u *_{g, m, p}^{n} v=&\sum_{i=0}^{\lfloor p\rfloor}(-1)^{i}
\binom{\lfloor m\rfloor+\lfloor n\rfloor-\lfloor p\rfloor-1+\delta_{\bar{m}}(r)+\delta_{\bar{n}}(T-r)+i}{i}\\
&\cdot \operatorname{Res}_{x} \frac{(1+x)^{\operatorname{wt} u+\lfloor m\rfloor+\delta_{\bar{m}}(r)+r / T-1}}{x^{\lfloor m\rfloor+\lfloor n\rfloor-\lfloor p\rfloor+\delta_{\bar{m}}(r)+\delta_{\bar{n}}(T-r)+i}}
Y(u, x) v,\end{align*}
if $\bar{p}-\bar{n}\equiv r\ \operatorname{mod}\ T$; and
$u *_{g, m, p}^{n} v=0$ otherwise.
Let
$$
O_{g, n, m}^{\prime}(V)=\operatorname{span}\left\{u \circ_{g, m}^n v \mid u, v \in V\right\}+L_{n, m}(V),
$$
where $L_{n, m}(V)=\operatorname{span}\{(L(-1)+L(0)+m-n) u \mid u \in V\}$ and for $u \in V^r, v \in V$,
$$
u \circ_{g, m}^n v=\operatorname{Res}_x \frac{(1+x)^{\wt u+\delta_{\bar{m}}(r)+\lfloor m\rfloor+r / T-1}}{x^{\lfloor m\rfloor+\lfloor n\rfloor+\delta_{\bar{m}}(r)+\delta_{\bar{n}}(T-r)+1}} Y(u, x) v .
$$
From \cite[Lemmas 3.3 and 3.4]{DJ1}, we respectively have
\begin{equation}\label{eq3.1}
\operatorname{Res}_x \frac{(1+x)^{\wt u+\delta_{\bar{m}}(r)+\lfloor m\rfloor+r / T+s-1}}{x^{\lfloor m\rfloor+\lfloor n\rfloor+\delta_{\bar{m}}(r)+\delta_{\bar{n}}(T-r)+k+1}} Y(u, x) v\in O^\prime_{g,n,m}(V)
\end{equation}
for any integers $k\geq s\geq 0$ and
\begin{equation}\label{eq3.2}
  Y(u, x) \mathbf{1} \equiv(1+x)^{-\wt u-m+n} u \bmod O^\prime_{g,n, m}(V).
\end{equation}
For any $a, b, c, u \in V$ and any $p_{1}, p_{2}, p_{3} \in (1/T)\N $, let $O_{g, n, m}^{\prime \prime}(V)$ be the linear span of
\begin{equation*}
u *_{g, m, p_{3}}^{n}\big((a *_{g, p_{1}, p_{2}}^{p_{3}} b) *_{g, m, p_{1}}^{p_{3}} c-a *_{g, m, p_{2}}^{p_{3}}(b *_{g, m, p_{1}}^{p_{2}} c)\big).
\end{equation*} Let
$$
O_{g, n, m}^{\prime \prime \prime}(V)=\sum_{p_{1}, p_{2} \in (1/T) \N}\left(V *_{g, p_{1}, p_{2}}^{n} O_{g, p_{2}, p_{1}}^{\prime}(V)\right) *_{g, m, p_{1}}^{n} V
$$
and
$$
O_{g, n, m}(V)=O_{g, n, m}^{\prime}(V)+O_{g, n, m}^{\prime \prime}(V)+O_{g, n, m}^{\prime \prime \prime}(V).
$$ Set  $A_{g,n,m}(V)=V/O_{g,n,m}(V)$ for any $m,n\in(1/T)\N.$ Then each $A_{g,n,m}(V)$ is an $A_{g,n}(V)\!-\!A_{g,m}(V)$-bimodule, as mentioned in our introduction.

Let $(M, Y_M)$ be a weak $g$-twisted $V$-module. For  any $n\in(1/T)\N$,  define $$\Omega_n(M, Y_M)=\{w\in M\mid u_{\wt u-1+i}w=0 \text{ for any } u\in V\text{ and }n<i\in(1/T)\Z\}.$$
As for the relation between  $O_{g,n,m}(V)$ and   $\Omega_n(M, Y_M)$,  the following holds  (see \cite[Lemma 5.2]{DJ1} or \cite[Remark 3.4\,(1)]{HXX}).
\begin{lem}\label{lem3.1}
	For any weak $g$-twisted $V$-module $M,$ $m,n \in (1/T)\N$ and $u\in O_{g, n, m}(V)$,
    we have ${\rm Res}_xx^{m-n-1}Y_M(x^{L(0)}u,x) =0$ on $\Omega_{m}(M,Y_M)$.
\end{lem}
The following result can be found in \cite[Lemma 3.1 and Proposition 4.2]{DJ1}.
\begin{lem}\label{lem3.2}
    	Let $m,n\in (1/T)\N$.
        \begin{itemize}\item[(1)] Let $ r\in\{0,1,\ldots, T-1\}$ with $r\equiv\bar{m}-\bar{n}\bmod T$.  Then
    	\[
    	O^{\prime}_{g,n,m}(V)=\bigoplus_{s\,\not\equiv\,{\bar m-\bar n}\bmod T} V^s\oplus (V^r\cap O^{\prime}_{g,n,m}(V))
    	\]
        and
        \[
    	O_{g,n,m}(V)=\bigoplus_{s\,\not\equiv\,{\bar m-\bar n}\bmod T} V^s\oplus (V^r\cap O_{g,n,m}(V)).
    	\]

        \item[(2)]  Set $\theta=e^{L(1)}(-1)^{L(0)}\in\End V$.  We have
        \[\theta(V^{i}\cap O_{g^{-1},m,n}(V))=V^i\cap O_{g,n,m}(V)\]
for any $i\in\{0,1,\cdots,T-1\}$.
\end{itemize}
    \end{lem}
    The following lemma is the key to showing our main result, which follows easily from Lemma \ref{lem3.2}.
\begin{lem}\label{lem3.3}
For any $n,m\in(1/T)\N$,
    then
    $
    O_{g,n,m}(V)=O^\prime_{g,n,m}(V)
    $
    if and only if
   $$
    \theta(V^{r}\cap O_{g^{-1},m,n}(V))\subseteq  O^\prime_{g,n,m}(V),
    $$
    where $r\in\{0,1,\cdots,T-1\}$ such that $r\equiv\bar{m}-\bar{n}\bmod T$.
\end{lem}

\section{Twisted regular representation $\mathfrak{D}(V)$}
Define a linear map
$$
Y^*(\cdot, x): V \rightarrow\left(\operatorname{End} V^*\right)\left[\left[x, x^{-1}\right]\right]
$$
by
$$
\left\langle Y^*(u, x) f, v\right\rangle=\left\langle f, Y\left(e^{x L(1)}\left(-x^{-2}\right)^{L(0)} u, x^{-1}\right) v\right\rangle
$$
for $u ,v\in V$ and  $f \in V^*\ (\text{the dual space of}\ V)$.

\begin{defi}\cite{LS1}
    Denote by $\mathfrak{D}(V)$ the set of all elements $f\in V^*$, satisfying the condition that for every $v \in V^r$ with $r\in\{0,1,\ldots, T-1\}$, there exists $k \in \mathbb{Z}$ such that
    \begin{equation}\label{eq4.1}
x^{-r/T}(x+1)^{k+r/T} Y^*(v, x)f \in V^*((x))
\end{equation}
\end{defi}

Let $u \in V^r$ with $r\in\{0,1,\cdots,T-1\}$ and $f \in \mathfrak{D}(V)$. Following \cite{LS1}, we define
$$
Y^{R}(u, x) f=\zeta^r(1+x)^{-k-r/T}\left((x+1)^{k+r/T} Y^*(u, x) f\right) \in x^{r/T}V^*((x)),
$$
where $k$ is any integer such that (\ref{eq4.1}) holds. The following is a special case of \cite[Theorem 3.6]{LS1}: taking $\sigma_1=g,\sigma_2=g^{-1},z=-1$.
\begin{thm}\label{g-1-weak}
$\left(\mathfrak{D}(V), Y^R\right)$ is a weak $g^{-1}$-twisted $V$-module.
    \end{thm}

\begin{rmk}\label{r-m-k-1}
Let $U$ be a vector space and $W$ its subspace.  Then $(U/W)^*$ is naturally viewed as a subspace of $U^*$  through the
quotient map $U\to U/W$.
Thus, if $u\in U$ satisfies $f(u)=0$ for any $f\in(U/W)^*$, then $u\in W.$
\end{rmk}

\begin{prop}\label{ex-porp}
Let $A^\prime_{g,n,m}(V)=V/O^\prime_{g,n,m}(V)$ for any
$m, n \in (1/T)\N$ and let $r\in\{0,1,\ldots, T-1\}$. Then for any  $f \in\left(A_{g,n, m}^{\prime}(V)\right)^*$ and for any $u \in V^r$ we have
$$
x^{\lfloor n\rfloor+\delta_{\bar{n}}(T-r)+\wt u-r/T}(x+1)^{\wt u+\delta_{\bar{m}}(r)+\lfloor m\rfloor+r / T-1}Y^*\left(u, x\right) f \in V^*[[x]].
$$ Furthermore, we have
$
\left(A_{g,n, m}^{\prime}(V)\right)^*\subseteq\Omega_{n}\left(\mathfrak{D}(V),Y^R\right)
$
as subspaces of $V^*$.
\end{prop}
\begin{proof}
Let $f \in\left(A_{g,n, m}^{\prime}(V)\right)^* \subset V^*$, i.e., $f \in V^*$ such that $\left.f\right|_{O_{g,n, m}^{\prime}(V)}=0$.  For any $k \in \mathbb{N}$ and $v\in V$, using (\ref{eq3.1}) we get
\begin{align*}
    &\Res_x x^{\lfloor n\rfloor+\delta_{\bar{n}}(T-r)+\wt u-r/T+k}(x+1)^{\wt u+\delta_{\bar{m}}(r)+\lfloor m\rfloor+r / T-1}\left\langle Y^*\left(u, x\right) f,v\right\rangle\\
    =\ &\Res_x x^{\lfloor n\rfloor+\delta_{\bar{n}}(T-r)+\wt u-r/T+k}(x+1)^{\wt u+\delta_{\bar{m}}(r)+\lfloor m\rfloor+r / T-1}
    \left\langle f, Y\left(e^{x L(1)}\left(-x^{-2}\right)^{L(0)} u, x^{-1}\right) v\right\rangle\\
    =\ &(-1)^{\wt u}\Res_x x^{\lfloor n\rfloor+\delta_{\bar{n}}(T-r)-\wt u-r/T+k}(x+1)^{\wt u+\delta_{\bar{m}}(r)+\lfloor m\rfloor+r / T-1}\left\langle f, Y\left(e^{x L(1)} u, x^{-1}\right) v\right\rangle\\
    =\ &
    (-1)^{\wt u}\operatorname{Res}_x \frac{(1+x)^{\wt u+\delta_{\bar{m}}(r)+\lfloor m\rfloor+r / T-1}}{x^{\lfloor m\rfloor+\lfloor n\rfloor+\delta_{\bar{m}}(r)+\delta_{\bar{n}}(T-r)+k+1}}\left\langle f, Y\left(e^{x^{-1} L(1)} u, x\right) v\right\rangle\\
    =\ &
    (-1)^{\wt u}\sum_{i \geq 0} \frac{1}{i!}\operatorname{Res}_x \frac{(1+x)^{\wt (L(1)^iu)+\delta_{\bar{m}}(r)+\lfloor m\rfloor+r / T+i-1}}{x^{\lfloor m\rfloor+\lfloor n\rfloor+\delta_{\bar{m}}(r)+\delta_{\bar{n}}(T-r)+k+i+1}}\left\langle f, Y(L(1)^iu, x) v\right\rangle
    = 0,
\end{align*}
proving  the first statement.
By this and  the definition of $Y^R(u, x) f$, we have
\begin{align}\label{ex-qe-00}
   &x^{\lfloor n\rfloor+\delta_{\bar{n}}(T-r)+\wt u-r/T}(1+x)^{\wt u+\delta_{\bar{m}}(r)+\lfloor m\rfloor+r / T-1}Y^R\left(u, x\right) f\\=\ &\zeta^r
x^{\lfloor n\rfloor+\delta_{\bar{n}}(T-r)+\wt u-r/T}(x+1)^{\wt u+\delta_{\bar{m}}(r)+\lfloor m\rfloor+r / T-1}Y^*\left(u, x\right) f\in \mathfrak{D}(V)[[x]], \nonumber
\end{align}
which implies
\begin{equation}\label{ex-qe-tion}
x^{\lfloor n\rfloor+\delta_{\bar{n}}(T-r)+\wt u-r/T}Y^R\left(u, x\right) f \in \mathfrak{D}(V)[[x]].\end{equation}
By Theorem \ref{g-1-weak}, we can write $Y^R(u,x)f$ as $\sum_{k\in -r/T+\Z}u^R_kx^{-k-1}f$. By this and \eqref{ex-qe-tion},
\[\mathfrak{D}(V)[[x]]\ni
x^{\lfloor n\rfloor+\delta_{\bar{n}}(T-r)+\wt u-r/T}Y^R\left(u, x\right) f=\sum_{k\in\Z}u^R_{\lfloor n\rfloor+\delta_{\bar{n}}(T-r)+\wt u-r/T+k}fx^{-k-1}.
\]
Thus, in particular, we have $
u^R_{\lfloor n\rfloor+\delta_{\bar{n}}(T-r)+\wt u-r/T+k}f=0$
for any $k\in\N$. Note that $$
n<\lfloor n \rfloor + \delta_{\bar{n}}(T - r)  - r/T+1= n + \delta_{\bar{n}}(T - r)  +(T-r-\bar{n})/T \leq n+1.
$$
Now in view of the definition of $\Omega_n(\mathfrak{D}(V), Y^R)$,   $f\in\Omega_n(\mathfrak{D}(V), Y^R)$, completing the proof of this proposition.
\end{proof}

The first statement of the following result follows from \cite[Remark 2.10]{L1} and \cite[Lemmas 2.6 and 2.7]{LS1}. The other statements can be proved similarly as \cite[Propositions 3.12 and 3.13]{L2}, respectively.
\begin{lem}\label{twi-rem}
Let $\left(E, Y_E\right)$ be a weak $g$-twisted $V$-module and $z_0 \in \mathbb{C}$. For $v \in V$, define
$$
Y_E^{\left[z_0\right]}(v, x)=Y_E\left(e^{-z_0\left(1+z_0 x\right) L(1)}\left(1+z_0 x\right)^{-2 L(0)} v, x /\left(1+z_0 x\right)\right)
$$
Then the pair $\left(E, Y_E^{\left[z_0\right]}\right)$ carries the structure of a weak $g$-twisted $V$-module and
$$
\operatorname{Res}_x x^m Y_E^{\left[z_0\right]}(v, x)=\operatorname{Res}_x x^m\left(1-z_0 x\right)^{2 \wt v-m-2} Y_E\left(e^{-z_0\left(1-z_0 x\right)^{-1} L(1)} v, x\right)
$$ for $m \in (1/T)\mathbb{Z}$. Moreover,
$
\Omega_n\left(E, Y_E\right)=\Omega_n\left(E, Y_E^{\left[z_0\right]}\right)
$ for any $n \in(1/T) \N$.
\end{lem}

\section{Proof of Theorem \ref{thm1.1}}
In this section we shall give the proof of Theorem \ref{thm1.1}.
Recall from \cite{FH1} that
$$
x^{-L(0)} L(1) x^{L(0)}=x L(1),
$$
which gives immediately the following  formula (see \cite{L1}):
$$
x^{-L(0)} e^{y L(1)} x^{L(0)}=e^{ xy L(1)}.
$$
In particular, we have
\begin{equation}\label{eq4.4}
e^{-x^2(x+1)^{-1}L(1)}(-x^{-2})^{L(0)}=(-x^{-2})^{L(0)}e^{(x+1)^{-1}L(1)}.
\end{equation}

Now we begin to present the proof of Theorem \ref{thm1.1}.

\begin{proof}
Let $p=n-m$ and let $r\in\{0,1,\ldots,T-1\}$ be such that $r\equiv\bar{m}-\bar{n}\bmod T$. Then
\begin{equation}\label{eq4.5}
 \delta_{\bar{m}}(r)+\lfloor m\rfloor+r / T+p=\xqz{n} + (\bar{n} + r - \bar{m})/T + \delta_{\bar{m}}(r) \in \N+1.
\end{equation}
Note that
$-1< -(T - 1)/T\le(\bar{n} + r - \bar{m})/T\in\Z.$
Thus, it is clear that  (\ref{eq4.5})  holds if
$
(\bar{n} + r - \bar{m})/T
\geq 1.
$
While for the case
$(\bar{n} + r - \bar{m})/T = 0$,
we have
$\bar{m} = \bar{n} + r \geq r$ and then
$\delta_{\bar{m}}(r) = 1,$
which implies  (\ref{eq4.5}).
Take any $v\in V^r$ and $f\in(A^\prime_{g,n,m}(V))^*$. It follows from Proposition \ref{ex-porp} and Lemma \ref{twi-rem} that
\begin{equation}\label{eq4.6}
f\in\Omega_n(\mathfrak{D}(V),Y^R)=\Omega_n(\mathfrak{D}(V),(Y^R)^{[-1]}).
\end{equation}  Set $q=-\xqz{m}-p-\delta_{\bar{m}}(r)-r/T$, which is an negative integer by (\ref{eq4.5}), and  set $s=\xqz{n}+|\xqz{p}|$. Then\begin{align*}
& \left\langle{\rm Res}_xx^{p-1}
\left(Y^R\right)^{[-1]}(z^{L(0)}v, )f,\mathbf{1}
\right\rangle\\
=&\operatorname{Res}_x x^{\wt v+p-1}\left\langle\left(Y^R\right)^{[-1]}\left( v, x\right) f, \mathbf{1}\right\rangle   \\
= & \operatorname{Res}_x x^{\wt v+p-1}(1+x)^{\wt v-p-1}\left\langle Y^R\left(e^{(1+x)^{-1} L(1)} v, x\right) f, \mathbf{1}\right\rangle \quad\quad\quad\quad\quad\quad{({\rm by\ Lemma}\ \ref{twi-rem})} \\
= & \operatorname{Res}_x \sum_{i=0}^{\infty}\binom{q}{i} x^{\wt v+p+i-1}(1+x)^{\wt v-p-q-1}\left\langle Y^R\left(e^{(1+x)^{-1} L(1)} v, x\right) f, \mathbf{1}\right\rangle  \\
= & \operatorname{Res}_x \sum_{k\geq 0}\frac{1}{k!}\sum_{i=0}^{\infty}\binom{q}{i} x^{\wt (L(1)^kv)+p+i+k-1}(1+x)^{\wt (L(1)^kv)-p-q-1}\left\langle Y^R\left(L(1)^k v, x\right) f, \mathbf{1}\right\rangle  \\
= & \operatorname{Res}_x \sum_{k\geq 0}\frac{1}{k!}\sum_{i=0}^{s}\binom{q}{i} x^{\wt (L(1)^kv)+p+i+k-1}(1+x)^{\wt (L(1)^kv)-p-q-1}\\
&\quad\times\left\langle Y^R\left(L(1)^k v, x\right) f, \mathbf{1}\right\rangle  \quad\quad\quad\quad\quad\quad\quad\quad\quad\quad\quad\quad\quad\quad\quad\quad\quad\quad \quad\quad\quad({\rm by}\ (\ref{eq4.6})) \\
=\ &\zeta^r \operatorname{Res}_x \sum_{k\geq 0}\frac{1}{k!}\sum_{i=0}^{s}\binom{q}{i} x^{\wt (L(1)^kv)+p+i+k-1}(x+1)^{\wt (L(1)^kv)-p-q-1} \\
&\quad\times\left\langle Y^*\left(L(1)^k v, x\right) f, \mathbf{1}\right\rangle  \quad\quad\quad\quad\quad\quad\quad\quad\quad\quad\quad\quad\quad\quad\quad\quad\quad\quad\quad\quad\quad({\rm by}\ (\ref{ex-qe-00})) \\
=\ &\zeta^r\operatorname{Res}_x \sum_{i=0}^{s}\binom{q}{i} x^{p+\wt v+i-1}(x+1)^{\wt v-p-q-1}
\left\langle Y^*\left(e^{(x+1)^{-1} L(1)} v, x\right) f, \mathbf{1}\right\rangle  \\
=\ &\zeta^r\operatorname{Res}_x \sum_{i=0}^{s}\binom{q}{i} x^{p+\wt v+i-1}(x+1)^{\wt v-p-q-1}\left\langle f, Y(e^{x L(1)}\left(-x^{-2}\right)^{L(0)} e^{(x+1)^{-1} L(1)} v, x^{-1}) \mathbf{1}\right\rangle  \\
=\ &\zeta^r\operatorname{Res}_x \sum_{i=0}^{s}\binom{q}{i} x^{p+\wt v+i-1}(x+1)^{\wt v-p-q-1}\\
&\quad\times\left\langle f, Y\left(e^{x(x+1)^{-1} L(1)}\left(-x^{-2}\right)^{L(0)} v, x^{-1}\right) \mathbf{1}\right\rangle \quad\quad\quad\quad\quad\quad\quad\quad\quad\quad\quad\quad\quad ({\rm by}\   (\ref{eq4.4}))\\
=\ &\zeta^r\operatorname{Res}_x \sum_{i=0}^{s}\binom{q}{i}\frac{(1+x)^{\wt v-p-q-1}}{x^{-q+i}}\left\langle f, Y\left(e^{(1+x)^{-1} L(1)}(-1)^{L(0)} v, x\right) \mathbf{1}\right\rangle  \\
=\ & \zeta^r\sum_{k \geq 0} \frac{1}{k!} \operatorname{Res}_x \sum_{i=0}^{s}\binom{q}{i}\frac{(1+x)^{\wt (L(1)^kv)-p-q-1}}{x^{-q+i}} \left\langle f, Y\left(L(1)^k(-1)^{L(0)} v, x\right) \mathbf{1}\right\rangle  \\
=\ & \zeta^r\sum_{k \geq 0} \frac{1}{k!} \operatorname{Res}_x \sum_{i=0}^{s}\binom{q}{i}\frac{(1+x)^{-q-1}}{x^{-q+i}} \left\langle f,L(1)^k(-1)^{L(0)} v\right\rangle \quad\quad\quad\quad\quad\quad\quad\quad\quad \quad({\rm by} \ (\ref{eq3.2})) \\
=\ & \zeta^r\sum_{k \geq 0} \frac{1}{k!} \left\langle f,L(1)^k(-1)^{L(0)} v
\right\rangle \quad\quad\quad\quad\quad\quad\quad\quad\quad\quad\quad\quad\quad\quad\quad \quad\quad\quad\quad\quad\quad({\rm by}\ \ref{eq4.5}))\\
=\ &\zeta^r\left\langle f,\theta(v)  \right\rangle \quad\quad\quad\quad\quad\quad\quad\quad\quad\quad\quad\quad\quad\quad\quad\quad\quad\quad\quad\quad\quad\quad\quad\quad({\rm see\ Lemma}\ \ref{lem3.2}\,(2))
\end{align*}
Thus  for any $u\in V^{r}\cap O_{g^{-1},m,n}(V)$, by Lemma \ref{lem3.1} and (\ref{eq4.6}),  we have
\begin{align*}
 0=\left\langle
{\rm Res}_xx^{p-1}\left(Y^R\right)^{[-1]}(x^{L(0)}u, x)f,\mathbf{1}
\right\rangle=\zeta^r\left\langle f,\theta(u) \right\rangle ,
\end{align*}
which implies $\theta(u)\in O^\prime_{g,n,m}(V)$ by Remark \ref{r-m-k-1}.  Then by Lemma \ref{lem3.3}, $O_{g,n,m}(V)=O^\prime_{g,n,m}(V)$, completing the proof.
\end{proof}

\section*{Acknowledgment}
J. Han is supported by the National Natural Science Foundation of China (No. 12271406).


\begin{thebibliography}{9999}
\bibitem{B} R. Borcherds, Vertex algebras, Kac-Moody algebras, and the monster, \emph{Pro.  Nat.  Acad.  Sci. USA} \textbf{83} (1986), 3068--3071.




\bibitem{DJ1}  C. Dong, C. Jiang, Bimodules and $g$-rationality of vertex operator algebras, \emph{Trans. Amer. Math. Soc.} \textbf{360} (2008), 4235--4262.





\bibitem{DLM3}  C. Dong, H.  Li, G. Mason, Vertex operator algebras and associative algebras, \emph{J. Algebra} \textbf{206} (1998), 67--96.

\bibitem{DLM1}  C.  Dong,  H.  Li,  G.  Mason,  Twisted  representations  of  vertex  operator  algebras,  \emph{Math.  Ann.}  \textbf{310} (1998), 571--600.

\bibitem{DLM2}  C. Dong, H.  Li, G. Mason, Twisted representations of vertex operator algebras and associative algebras, \emph{Internat. Math. Res. Notices} \textbf{8} (1998), 389--397.





\bibitem{FH1}  I. Frenkel, Y. Huang, J. Lepowsky, On axiomatic apporoaches to vertex operator algebras and modules, \emph{Mem. Amer. Math. Soc.} \textbf{104} (1993), 1--64.

\bibitem{FLM1}  I. Frenkel, J. Lepowsky, A. Meurman, Vertex Operator Algebras and the Monster, Pure Appl. Math. \textbf{134}, Academic Press, Boston, 1988.


\bibitem{FZ1}  I.  Frenkel,  Y.  Zhu,  Vertex  operator  algebras  associated  to  representations  of  affine  and  Virasoro algebras, \emph{Duke Math. J.} \textbf{66} (1992), 123--168.


\bibitem{H}  J. Han, Bimodules and universal enveloping algebras associated to VOAs, \emph{Israel J. Math.} \textbf{247} (2022), 905--922.

\bibitem{HX} J. Han, Y. Xiao, Associative algebras and universal enveloping algebras associated to VOAs, \emph{J. Algebra}  \textbf{564} (2020), 489--498.

\bibitem{HXX}
J. Han, Y. Xiao, S. Xu,
Twisted bimodules and universal enveloping
algebras associated to VOAs, \emph{J. Algebra}  \textbf{664} (2025), 1--25.

\bibitem{He}  X. He, Higher level Zhu algebras are subquotients of universal enveloping algebras, \emph{J. Algebra}  \textbf{491} (2017), 265--279.







\bibitem{L} H. Li, Determining fusion rules by $A(V)$-modules and bimodules, {\em J. Algebra} {\bf212} (1999), 515--556.

\bibitem{L1}
H. Li, Regular representations of vertex operator algebras, {\em Commun. Contemp. Math.}  \textbf{4} (2002),
639--683.

\bibitem{L2}
H. Li, The regular representations and $A_n(V )$-algebras, {\em CRM Proc. Lect. Notes} \textbf{30} (2001), 99--116.


\bibitem{LS1}
H. Li, J. Sun, Twisted regular representations for vertex operator algebras, {\em J. Algebra}  \textbf{629}
 (2023), 124--161.






 
 \bibitem{ZY}Y. Zhu, Twisted regular representations and bimodules in vertex operator algebra theory, arXiv:2505.16464.


\bibitem{Z} Y. Zhu,  Modular invariance of characters of vertex operator algebras, {\it J. Amer. Math. Soc.} {\bf9} (1996),  237--302.

\end{thebibliography}
\end{document}